\documentclass[11pt, reqno]{amsart}
\usepackage[all]{xy}

\usepackage{amssymb}
\usepackage{amsthm}
\usepackage{amsmath,cancel}
\usepackage{amscd,enumitem}
\usepackage{verbatim}
\usepackage{float}
\usepackage{geometry}
\usepackage{color,mdframed}
\usepackage{bbm}
\usepackage[mathscr]{eucal}
\usepackage[all]{xy}
\usepackage{bm}

\usepackage{a4wide,fullpage}
\usepackage[draft]{hyperref}
\usepackage{mathtools}
\usepackage{calligra}
\usepackage{stmaryrd,float}
\DeclareMathAlphabet{\mathcalligra}{T1}{calligra}{m}{n}

\newtheorem{theorem}{Theorem}
\newtheorem{corollary}[theorem]{Corollary}
\newtheorem{lemma}[theorem]{Lemma}
\newtheorem{conjecture}[theorem]{Conjecture}

\newtheorem{proposition}[theorem]{Proposition}
\newtheorem*{proposition*}{Proposition}
\newtheorem*{theorem*}{Theorem}

\theoremstyle{definition}
\newtheorem*{remark}{Remark}

\theoremstyle{remark}

\newcommand{\Q}{\mathbb{Q}}
\newcommand{\Z}{\mathbb{Z}}
\newcommand{\N}{\mathbb{N}}

\renewcommand{\H}{\mathbb{H}}

\newcommand{\ord}{{\text {\rm ord}}}

\newcommand{\xuparrow}[1]{%
	{\left\uparrow\vbox to #1{}\right.\kern-\nulldelimiterspace}
}

\newcommand{\SL}{{\text {\rm SL}}}

\renewcommand{\pmod}[1]{\  \,  \left(  \operatorname{mod} \,  #1 \right)}
\newcommand{\lcm}{\operatorname{lcm}}

\DeclareMathAlphabet{\mathpzc}{OT1}{pzc}{m}{it}

\def\lp{\left(}
\def\rp{\right)}

\def\a{\alpha}

\def\d{\delta}

\numberwithin{equation}{section}
\numberwithin{theorem}{section}
\allowdisplaybreaks

\title{Conjectures of Sun about sums of polygonal numbers}
\author{Kathrin Bringmann}
\address{Department of Mathematics and Computer Science\\Division of Mathematics\\University of Cologne\\ Weyertal 86-90 \\ 50931 Cologne \\Germany}
\email{kbringma@math.uni-koeln.de}
\author{Ben Kane}
\address{Department of Mathematics\\ University of Hong Kong\\ Pokfulam, Hong Kong}
\email{bkane@hku.hk}
\date{\today}
\thanks{The first author has received funding from the European Research Council (ERC) under the European Union's Horizon 2020 research and innovation programme (grant agreement No. 101001179). The research of the second author was supported by a grant from the Research Grants Council of the Hong Kong SAR, China (project number HKU 17303618). }
\begin{document}	
\begin{abstract}
In this paper, we show that certain sums of generalized $m$-gonal numbers represent every positive integer if and only if they represent every positive integer up to an explicit bound $C_m$, verifying a conjecture of Sun for sufficiently large positive integers.
\end{abstract}

	\maketitle

\section{Introduction and statement of results}\label{sec:intro}
 For $m\in\N_{\geq 3}$ and $\ell\in\Z$, let $p_m(\ell)$ be the \begin{it}$\ell$-th (generalized) $m$-gonal number\end{it}
\begin{equation*}
p_m(\ell):=\frac12 (m-2)\ell^2-\frac12 (m-4)\ell.
\end{equation*}
For $\ell\in\N$, these count the number of dots contained in a regular polygon with $m$ sides having $\ell$ dots on each side. 
For example, the special case $m=3$ corresponds to  triangular numbers, $m=4$ gives squares, and $m=5$ corresponds to pentagonal numbers. 
There are several conjectures related to sums of polygonal numbers. 
Specifically, for $\bm{\alpha}\in\N^d$,\footnote{We denote vectors like $\bm{\alpha}$ in bold and the $j$-th component of a vector $\bm{\alpha}$ we write as $\alpha_j$ throughout.} we are interested in the solvability of the Diophantine equation
\begin{equation}\label{eqn:genpolysum}
\sum_{1 \leq j \leq d} \alpha_j p_m(\ell_j)=n
\end{equation}
with $\ell_j\in \N_0$ or $\ell_j\in\Z$. We call such a sum \begin{it}universal\end{it} if it is solvable for every $n\in\N$. Fermat stated (his claimed proof was not found in his writings) that every positive integer is the sum of three triangular number, four squares, five pentagonal numbers, and in general at most $m$ $m$-gonal numbers. In other words, he claimed that the sum $\sum_{1\leq j\leq m} p_m(\ell_j)$ is universal. His claim for squares ($m=4$) was proven by Lagrange in 1770, the claim for triangular numbers ($m=3$) was shown by Gauss in 1796, and the full conjecture was proven by Cauchy in 1813. 
Going in another direction, Ramanujan fixed $m=4$ and conjectured a full list of choices of $\bm{\alpha}\in\N^4$ for which the resulting sum is universal; this conjecture was later proven by Dickson \cite{Dickson}. Following this, the classification of universal quadratic forms was a central area of study throughout the 20th century, culminating in the Conway--Schneeberger 15-theorem \cite{Bhargava,ConwaySchneeberger} and the 290-theorem \cite{BhargavaHanke}, which state that arbitrary quadratic forms whose cross terms are even (resp. are allowed to be odd) are universal if and only if they represent every integer up to $15$ (resp. $290$). Theorems of this type are now known as \begin{it}finiteness theorems\end{it}. Namely, given an infinite set $S\subseteq\N$, one determines a finite subset $S_0$ of $S$ such that a solution to \eqref{eqn:genpolysum} exists for every $n\in S$ if and only if it exists for every $n\in S_0$. Taking $S=\N$, one obtains a condition for universality of a given sum of polygonal numbers. For example, choosing $m=3$ or $m=6$, \eqref{eqn:genpolysum} is solvable with $\bm{\ell}\in\Z^d$ for all $n\in\N$ if and only if it is solvable for every $n\leq 8$ \cite{BosmaKane}, for $m=5$ it is solvable with $\bm{\ell}\in\Z^d$ for all $n\in\N$ if and only if it is solvable for every $n\leq 109$ \cite{Ju}, while it is solvable with $\bm{\ell}\in\N_0^d$ for all $n\in\N$ if and only if it is solvable for every $n\leq 63$ \cite{JuKim} and for $m=8$ it is solvable for $\bm{\ell}\in\Z^d$ for all $n\in\N$ if and only if it is solvable for every $n\leq 60$ \cite{JuOh}.

Here we consider the question of universality in the case $\bm{\alpha}=(1,2,4,8)$ as one varies $m$. Specifically, we have the following conjecture of  Sun  (see \cite[Conjecture 5.4]{SunLagrange}).
\begin{conjecture}\label{conj:Sun1248}
For $m\in\{7,9,10,11,12,13,14\}$ and $\bm{\ell}\in\Z^4$, the equation
\begin{equation}\label{eqn:polygonalrep}
p_m(\ell_1)+ 2p_m(\ell_2)+4p_m(\ell_3)+ 8p_m(\ell_4) = n
\end{equation}
is solvable for every $n\in\N$. 
\end{conjecture}

\begin{remark}
	A proof of Conjecture \ref{conj:Sun1248} would give a classification of those $m$ for which the sum \eqref{eqn:genpolysum} is universal in the case $\bm{\alpha}=(1,2,4,8)$. By direct computation, one sees that $p_m(\ell)\in\{0,1\}$ or $p_m(\ell)\geq m-3$. Using this, one obtains that \eqref{eqn:polygonalrep} is not solvable for $n=16$ for every $m\geq 20$. For $m=16$, $m=17$, $m=18$, and $m=19$, one finds directly that there is no solution for $n=29$, $n=30$, $n=16$, and $n=17$, respectively. Moreover, for $m\in\{3,6\}$ it is known by work of Liouville \cite{Liouville} that the sum is universal, for $m=4$ it was conjectured by Ramanujan and proven by Dickson \cite{Dickson} that the sum is universal, while for $m=5$ (resp. $m=8$) it was shown by Sun in \cite{SunTernary} (resp. \cite{SunLagrange}) to be universal.
\end{remark}

In this paper, we prove that Conjecture \ref{conj:Sun1248} is true for $n$ sufficiently large.

\begin{theorem}\label{thm:Sun1248sufflarge}
	For $m\in\{7,9,10,11,12,13,14\}$, there exists an explicit constant $C_m$ (defined in Table \ref{tab:finalbounds}) such that \eqref{eqn:polygonalrep} is solvable with $\bm{\ell}\in\Z^4$ for every $n\in\N_{\geq C_m}$, with the restriction that $n\not\equiv 4\pmod{16}$ if $m=12$. 
\end{theorem}

\begin{remark}
To prove Conjecture \ref{conj:Sun1248} for $m=12$, it suffices to show that \eqref{eqn:polygonalrep} holds for all $n\in\N$ with $n\not\equiv 4\pmod{16}$ (see Lemma \ref{lem:p122power}). Hence the restriction in Theorem \ref{thm:Sun1248sufflarge} is natural.
\end{remark}

By completing the square, one easily sees that representations of integers as sums of polygonal numbers are closely related to sums of squares with congruence conditions. In particular, setting 
\begin{align*}
r_{m,\bm{\alpha}}(n):&=\#\left\{\bm{\ell}\in\Z^4: \sum_{1\leq j\leq 4} \alpha_j p_{m}(\ell_j)=n\right\},
\\
s_{r,M,\bm{\alpha}}(n):&=\#\left\{ \bm{x}\in\Z^4: \sum_{1\leq j\leq 4} \alpha_j x_j^2=n,\ x_j\equiv r\pmod{M}\right\},
\end{align*}
we have  
\begin{equation}\label{eqn:r*s*rel}
r_{m,(1,2,4,8)}(n)= s_{m,2(m-2),(1,2,4,8)}\left(8(m-2)n + 15(m-4)^2\right).
\end{equation}
Hence since Conjecture \ref{conj:Sun1248} is equivalent to proving that $r_{m,(1,2,4,8)}(n)>0$ for every $n\in\N$ and $m\in \{7,9,10,11,$ $12,13,14\}$, the conjecture is equivalent to showing that for every $n\in\N$ we have
\begin{equation}\label{eqn:conjectureequivalent}
s_{m,2(m-2),(1,2,4,8)}\left(8(m-2)n + 15(m-4)^2\right)>0. 
\end{equation}

We investigate the numbers $s_{r,M,\bm{\alpha}}(n)$ by forming the generating function (setting $q:=e^{2\pi i \tau}$)
\[
\Theta_{r,M,\bm{\alpha}}(\tau):=\sum_{n\geq 0} s_{r,M,\bm{\alpha}}(n) q^{n}.
\]
 It is well-known that these functions are modular forms (see Lemma \ref{lem:ThetaModular} for the precise statement). By the theory of modular forms, there is a natural decomposition
\begin{equation}\label{eqn:thetadecompose}
\Theta_{r,M,\bm{\alpha}}=E_{r,M,\bm{\alpha}} + f_{r,M,\bm{\alpha}},
\end{equation}
where $E_{r,M,\bm{\alpha}}$ lies in the space spanned by Eisenstein series and $f_{r,M,\bm{\alpha}}$ is a cusp form. In order to prove Theorem \ref{thm:Sun1248sufflarge}, we obtain in the special case $r=m$, $M=2(m-2)$, and $\bm{\alpha}=(1,2,4,8)$ an explicit lower bound for the $n$-th Fourier coefficient $a_{r,M,\bm{\alpha}}(n)$ of $E_{r,M,\bm{\alpha}}$ in Corollary \ref{cor:EisensteinLower} and an explicit upper bound on the absolute value of the $n$-th Fourier coefficient $b_{r,M,\bm{\alpha}}(n)$ of $f_{r,M,\bm{\alpha}}$ in the proof of Theorem \ref{thm:Sun1248sufflarge}. 

The paper is organized as follows.  In Section \ref{sec:prelim}, we recall properties of the theta functions $\Theta_{r,M,\bm{\alpha}}$, the actions of certain operators on modular forms, the decomposition of modular forms into Eisenstein series and cusp forms, and evaluate certain Gauss sums. In Section \ref{sec:growth}, we investigate the growth of the theta functions towards all cusps and use this to compute the Eisenstein series component of the decomposition \eqref{eqn:thetadecompose}. The Fourier coefficients of the Eisenstein series components are then explicitly computed and lower bounds are obtained in Section \ref{sec:Eisenstein}. We complete the paper by obtaining upper bounds on the coefficients of the cuspidal part of the decomposition \eqref{eqn:thetadecompose} and prove Theorem \ref{thm:Sun1248sufflarge} in Section \ref{sec:mainproof}.

\section{Setup and preliminaries}\label{sec:prelim}

\subsection{Modularity of the generating functions}
In this subsection, we consider the modularity properties of the theta functions $\Theta_{r,M,\bm{\alpha}}$. To set notation, for $\Gamma_1(N)\subseteq\Gamma\subseteq\SL_2(\Z)$ ($N\in\N$) and a character $\chi$ modulo $N$, let $M_{k}(\Gamma,\chi)$ be the space of modular forms of weight $k$ with character $\chi$. In particular an element $f$ in this space satisfies, for $\gamma=\left(\begin{smallmatrix}a&b\\c&d\end{smallmatrix}\right)\in \Gamma$,
\[
f\big|_k\gamma(\tau):=(c\tau+d)^{-k}f(\gamma\tau)=\chi(d)f(\tau).
\]
Setting $\Gamma_{N,L}:=\Gamma_0(N)\cap\Gamma_1(L)$, by \cite[Theorem 2.4]{Cho}, we have the following.

\begin{lemma}\label{lem:ThetaModular}
For $\bm{\alpha}\in \N^4$, we have
\[
\Theta_{r,M,\bm{\alpha}} \in M_{2}\left(\Gamma_{4\operatorname{lcm}(\bm{\alpha})M^2,M}, \left(\frac{\prod_{j=1}^4 \alpha_j}{\cdot}\right)\right).
\] 
\end{lemma}

\subsection{Operators on non-holomorphic modular forms}

For a translation-invariant function $f$ with Fourier expansion (denoting $\tau=u+iv\in\H$)
\[
f(\tau)=\sum_{n\ge0} c_{f,v}(n) q^n,
\]
we define the \textit{sieving operator} ($M$, $m\in\N$)
\begin{equation*}
f\big|S_{M,m}(\tau):=\sum_{\substack{n\ge0\\ n\equiv m\pmod{M}}} c_{f,v}(n) q^n.
\end{equation*}
As usual, we also define the \textit{$V$-operator} ($\d\in\N$) by
\[
f\big|V_{\delta}(\tau):=\sum_{n\geq 0} c_{f,\delta v}(n)q^{\delta n}.
\]
We require the modularity properties of (non-holomorphic) modular forms under the operators $S_{M,m}$ and $V_d$. 
Arguing via commutator relations for matrices, a standard argument (for example, see the proof of \cite[Lemma 2]{LiWinnie}), one obtains the following.
\begin{lemma}\label{lem:modularoperators}
Suppose that $k\in\Z$, $L,N\in\N$ with $L\mid N$, and $f$ satisfies weight $k$ modularity on $\Gamma_{N,L}$.
\begin{enumerate}[leftmargin=*,label={\rm(\arabic*)}]
\item
For $d\in \N$, the function $f|V_d$ satisfies weight $k$ modularity on $\Gamma_{\lcm(4,Nd),L}$.
\item
For $m\in\Z$ and $M\in\N$, the function $f|S_{M,m}$ satisfies weight $k$ modularity on $\Gamma_1(NM^2)$.
\end{enumerate}
\end{lemma}
It is useful to determine the commutator relations between the $V$-operator and sieving.
\rm
\begin{lemma}\label{lem:VSieveCommute}
Let $m\in\Z$ and $M_1,M_2\in\N$ be given and set $d:=\gcd(M_1,M_2)$ and $\mu_j:=\frac{M_j}{d}$. Then for any translation-invariant function $f$ we have
\[
f\big|V_{M_1}\big|S_{M_2,m}=
\begin{cases}
f\big|S_{\mu_2,\overline{\mu}_1\frac{m}{d}}\big|V_{M_1}&\text{if }d\mid m,\\
0&\text{otherwise},
\end{cases}
\]
 where $\bar{\mu}_1$ is the inverse of $\mu_1 \pmod{\mu_2}$. 
\end{lemma}

\subsection{Decomposition into Eisenstein series and cusp forms}\label{sec:decomposition}

Comparing Fourier coefficients on both sides of \eqref{eqn:thetadecompose}, we have 
\begin{equation}\label{eqn:sabdecompose}
s_{r,M,\bm{\alpha}}(n)=a_{r,M,\bm{\alpha}}(n)+b_{r,M,\bm{\alpha}}(n).
\end{equation}
Theorem \ref{thm:Sun1248sufflarge} is equivalent to showing \eqref{eqn:conjectureequivalent} for $n$ sufficiently large (with the restriction that $n\not\equiv 4\pmod{16}$ for $m=12$). Roughly speaking, the approach in this paper to proving \eqref{eqn:conjectureequivalent} is to prove that for $n$ sufficiently large with $n\equiv 15(m-4)^2\pmod{8(m-2)}$ (noting the congruence conditions in \eqref{eqn:conjectureequivalent})
\[
a_{r,M,\bm{\alpha}}(n)>|b_{r,M,\bm{\alpha}}(n)|.
\]

To obtain an upper bound for $|b_{r,M,\bm{\alpha}}(n)|$, we recall that Deligne \cite{Deligne} proved that for a normalized newform $f(\tau)=\sum_{n\geq 1} c_f(n)q^n$ of weight $k$ on $\Gamma_0(N)$ with Nebentypus character $\chi$ (normalized so that $c_f(1)=1$), we have 
\begin{equation}\label{eqn:Deligne}
|c_f(n)|\leq \sigma_0(n) n^{\frac{k-1}{2}},
\end{equation}
where $\sigma_k(n):=\sum_{d\mid n} d^k$. To obtain an explicit bound for $|c_f(n)|$ for arbitrary $f\in S_{k}(\Gamma_1(N))$, we combine \eqref{eqn:Deligne} with a trick implemented by Blomer \cite{Blomer} and Duke \cite{DukeTernary}.  For cusp forms $f,g\in S_{k}(\Gamma)$, we define the \begin{it}Petersson inner product\end{it} by
\[
\left<f,g\right>:=\frac{1}{[\SL_2(\Z):\Gamma]}\int_{\Gamma\backslash\H} f(\tau)\overline{g(\tau)} v^k \frac{dudv}{v^2}.
\]
Letting $\|f\|:=\sqrt{\left<f,f\right>}$ denote the {\it Petersson norm} of $f\in S_k(\Gamma)$, a bound for $|c_f(n)|$ in terms of $\|f\|$ may be obtained.  Specifically, suppose that $f$ is a cusp form $f$ of weight $k\in\N$ on  $\Gamma_{N,L}$ (with $L\mid N$) and character $\chi$ modulo $N$. Using Blomer's method from \cite{Blomer}, an explicit bound is obtained in \cite[Lemma 4.1]{BanerjeeKane} for $|c_f(n)|$ as a function of $N$, $L$, and the Petersson norm $\|f\|$. Denoting by $\varphi$ Euler's totient function, we recall a bound from the case $k=2$ below (see \cite[(4.4)]{BanerjeeKane}).
\begin{lemma}\label{lem:BanerjeeKane}
Suppose that $f\in S_{2}(\Gamma_{N,L},\chi)$ with $L\mid N$ and $\chi$ a character modulo $N$. Then we have the inequality 
\begin{equation*}
\left|c_{f}(n)\right|\leq 6.95\cdot 10^{18}\cdot \|f\| N^{1+2.5\cdot 10^{-6}}\prod_{p\mid N}\left(1+\frac{1}{p}\right)^{\frac{1}{2}}\varphi(L) n^{\frac{3}{5}}.
\end{equation*}
\end{lemma}

By Lemma \ref{lem:BanerjeeKane}, in order to obtain an explicit bound for $|b_{r,M,\bm{\alpha}}(n)|$, it remains to estimate $\|f_{r,M,\bm{\alpha}}\|$, where $f_{r,M,\bm{\alpha}}$ is the cusp form appearing in the decomposition in \eqref{eqn:thetadecompose}.
An explicit bound for $\|f_{r,M,\bm{\alpha}}\|$ was obtained in \cite[Lemma 3.2]{KKT}. To state the result, let $\bm{\alpha}\in\Z^{\ell}$. For the quadratic form $Q=Q_{\bm{\alpha}}$ given by 
\[
Q_{\bm{\alpha}}(\bm{x}):=\sum_{j=1}^{\ell} \alpha_j x_j^2,
\]
we define the \textit{level} and the \textit{discriminant} of $Q_{\bm \alpha}$ as
\[
N_{\bm{\alpha}}=4\lcm(\bm{\alpha}),\qquad D_{\bm{\alpha}}=2^{\ell}\prod_{j=1}^{\ell} \alpha_j. 
\]
\begin{lemma}\label{lem:KKT}
Let $\ell\geq 4$ be even, $\bm{\alpha}\in\N^{\ell}$, $r\in\Z$, and $M\in\N$. Then
\begin{multline*}
\|f_{r,M,\bm{\alpha}}\|^2\leq \frac{3^{2\ell-2}\left(\frac{\ell}{2}-2\right)!}{2^{\frac{\ell}{2}-3}\pi^{\ell}} \frac{M^{2\ell-4}N_{\bm{\alpha}}^{\ell-2}}{\prod_{p\mid M^2N_{\bm{\alpha}}}\left(1-p^{-2}\right)}\sum_{\delta\mid M^2N_{\bm{\alpha}}}\varphi\left(\frac{M^2N_{\bm{\alpha}}}{\delta}\right)\varphi(\delta)\frac{M^2N_{\bm{\alpha}}}{\delta}\left(\frac{\gcd(M^2,\delta)}{M^2}\right)^{\ell} \\
\times \sum_{m=0}^{\frac{\ell}{2}-2} \frac{(2\pi)^{-m}}{\left(\frac{\ell}{2}-2-m\right)!}(\ell-m-2)! \left(\frac{9}{ D_{\bm{\alpha}}}(\ell-m-1)\frac{M^2N_{\bm{\alpha}}}{\pi} +\ell^2 \right).
\end{multline*}
\end{lemma}

\subsection{Gauss sums}

Define the \begin{it}generalized quadratic Gauss sum\end{it} ($a,b\in\Z$, $c\in\N$)
\begin{equation*}
G(a,b;c):=\sum_{\ell\pmod{c}}e^{\frac{2\pi i}{c}\left(a\ell^2+b\ell\right)}.
\end{equation*}
Background information and many properties of these sums may be found in \cite{BerndtEvansWilliams}. 
To state the properties that we require, for $d$ odd, we define
\[
\varepsilon_d:=\begin{cases} 1&\text{if $d\equiv 1\pmod{4}$,}\\ i&\text{if $d\equiv 3\pmod{4}$,}\end{cases} 
\]
and we write $[a]_{b}$ for the inverse of $a$ modulo $b$ if $\gcd(a,b)=1$. 

\begin{lemma}\label{lem:GaussSums}
For $a,b\in\Z$ and $c,d\in\N$, the following hold. 
 \noindent

\noindent
\begin{enumerate}[leftmargin=*,label={\rm(\arabic*)}]
\item If $\gcd(a,c)\nmid b$, then $G(a,b;c)=0$, while if $\gcd(a,c)\mid b$ then 
\begin{equation*}\label{eqn:Gabcgcd}
G(a,b;c)=\gcd(a,c) G\left(\frac{a}{\gcd(a,c)},\frac{b}{\gcd(a,c)};\frac{c}{\gcd(a,c)}\right).
\end{equation*}
\item If $\gcd(a,c)=1$ and $c$ is odd, then 
\begin{equation*}\label{evaluateG}
G(a,b;c) = \varepsilon_c \sqrt{c} \left(\frac ac\right) e^{-\frac{2\pi i [4a]_c b^2}{c}}.
\end{equation*}
\item
If $\gcd(c,d)=1$, then 
\[
G(a,b;cd)=G(ac,b;d)G(ad,b;c).
\]
\item If $\gcd(a,c)=1$, $4\mid c$, and $b$ is odd, then $G(a,b;c)=0$.
\item If $a$ is odd, $b$ is even, and $r\in\N_{\geq 2}$, then  
\[
G\left(a,b;2^r\right)= 2^{\frac{r}{2}} (1+i) \left(\frac{-2^{r}}{a}\right) \varepsilon_a e^{\frac{2\pi i}{2^{r}} \left(-[a]_{2^{r}} \frac{b^2}{4}\right)}.
\]
\end{enumerate}
\end{lemma}

We require an explicit evaluation of certain Gauss sums that naturally occur in the study of theta functions (see Lemma \ref{lem:Thetacuspgrowth} below). Throughout the paper, for $k,M\in\N$, and $r\in\Z$ with $\ord_2(r)\leq \ord_2(M)$, we write $M=2^{\mu}M_0$, $r=2^{\varrho}r_0$ (with $\varrho\leq \mu$), and $k=2^{\kappa}k_0$ with $M_0$, $r_0$, and $k_0$ odd. We furthermore set $g_0:=\gcd(M_0,k_0)$ and $g_1:=\gcd(g_0,\frac{k_0}{g_0})$.

\begin{lemma}\label{lem:G(a,b,c)eval}
Suppose that $h\in\Z$, $k\in\N$ with $\gcd(h,k)=1$, $\ell\in\N_0$,  $r\in\Z$, $M\in\N$ with $\gcd(M,r)\in\{1,2,4\}$, and $\varrho\leq \mu$.
\begin{enumerate}[leftmargin=*, label=\rm(\arabic*)]
\item If $g_1\neq 1$ or $\varrho < \min(\mu,\kappa-\ell -\mu)-1$, then
\[
G\left( 2^{\ell}h M^2,2^{\ell+1}hrM;k\right)=0.
\]

\item Suppose that $g_1=1$ and $\varrho \geq \min(\mu, \kappa-\ell-\mu)-1$. Setting $\delta:=\min(\ell+2\varrho,\kappa)$, we then have
\begin{multline*}
\quad e^{\frac{2\pi i 2^{\ell}hr^2}{k}}G\left(2^{\ell} h M^2,2^{\ell+1}hrM;k\right)\\
=\sqrt{k g_0}\begin{cases}
2^{\frac{\kappa}{2}}  \varepsilon_{\frac{k_0}{g_0}}   \left(\frac{2^{\ell+\kappa} h g_0}{\frac{k_0}{g_0}}\right) e^{\frac{2\pi i hr_0^2}{2^{\kappa-\delta}g_0}  2^{\ell+2\varrho-\delta} \left[\frac{k_0}{g_0}\right]_{2^{\kappa-\delta}g_0}}&\begin{array}{l}\hspace{-.2cm}\text{if }\kappa\leq \ell+2\mu\\ \quad\text{or }\kappa=\ell+2\mu+1\text{ and }\varrho=\mu-1,\end{array}\\
2^{\frac{\ell+2\mu}{2}}(1+i)  \varepsilon_{hg_0} \left(\frac{-2^{\ell+\kappa}\frac{k_0}{g_0}}{hg_0}\right)e^{\frac{2\pi i hr_0^2}{g_0}\left[2^{\kappa-\ell-2\mu}\frac{k_0}{g_0}\right]_{g_0}}&\text{if }\kappa \geq \ell+2\mu+2\text{ and }\varrho=\mu,\\
0&\text{otherwise.}
\end{cases}
\end{multline*}
\end{enumerate}
\end{lemma}
\begin{proof}
We evaluate $G(a,b;c)$ for $a:= 2^{\ell} h M^2$, $b:=2^{\ell+1}hrM$, and $c:=k$. By Lemma \ref{lem:GaussSums} (1), $G(a,b;c)=0$ unless $\gcd(a,c)\mid b$. Hence we first compute, using the fact that $\gcd(h,k)=1$, $\gcd(\frac{M_0}{g_0},\frac{k_0}{g_0})=1$, and $\frac{k_0}{g_0}$ is odd, 
\begin{equation}\label{eqn:gcd}
\gcd(a,c) =  2^{\min(\ell+2\mu,\kappa)}g_0g_2,
\end{equation}
where $g_2:=\gcd(M_0,\frac{k_0}{g_0})$.

\noindent
(1) A direct calculation gives that $\gcd(a,c)\mid b$ if and only if $g_1=1$ and $\varrho \geq \min(\mu,\kappa-\ell-\mu)-1$, which implies the claim by Lemma \ref{lem:GaussSums} (1).

\noindent(2) Set $\gamma:=\min(\ell+2\mu,\kappa)$. Note that  $\gamma\leq \ell+\mu+\varrho+1$. From the calculation yielding (1), we see that $g_1=1$ implies $g_2=1$. Plugging $g_1=g_2=1$ into \eqref{eqn:gcd} yields $\gcd(a,c)=2^{\gamma}g_0$ and it is not hard to see that $\gcd(a,c)\mid b$.  Therefore Lemma \ref{lem:GaussSums} (1),(3) implies that 
\begin{multline*}
 G(a,b;c)=2^{\gamma} g_0G\left(2^{\ell+2\mu+\kappa-2\gamma} hM_0 \frac{M_0}{g_0}, 2^{\ell+\mu+\varrho+1-\gamma} hr_0\frac{M_0}{g_0}; \frac{k_0}{g_0}\right)\\
\times G\left(2^{\ell+2\mu-\gamma} h M_0 \frac{M_0}{g_0}\frac{k_0}{g_0}, 2^{\ell+\mu+\varrho+1-\gamma} hr_0\frac{M_0}{g_0}; 2^{\kappa-\gamma}\right).
\end{multline*}
 Since $\frac{k_0}{g_0}$ is odd, we use Lemma \ref{lem:GaussSums} (2) to evaluate the first Gauss sum, yielding, after simplification, 
\begin{align}\nonumber
G(a,b;c)=2^{\gamma}  \varepsilon_{\frac{k_0}{g_0}}  \sqrt{k_0 g_0} &\left(\frac{2^{\ell+\kappa} h g_0}{\frac{k_0}{g_0}}\right) e^{-\frac{2\pi i  hr_0^2}{\frac{k_0}{g_0}} 2^{\ell+2\varrho}\left[2^{\kappa}g_0\right]_{\frac{k_0}{g_0}}}\\
\label{eqn:tosplit2pow}
&\hspace{1cm}\times G\left(2^{\ell+2\mu-\gamma} h M_0 \frac{M_0}{g_0}\frac{k_0}{g_0}, 2^{\ell+\mu+\varrho+1-\gamma} hr_0\frac{M_0}{g_0}; 2^{\kappa-\gamma}\right). 
\end{align}
It remains to evaluate the final Gauss sum in \eqref{eqn:tosplit2pow}. We use Lemma \ref{lem:GaussSums} (4) and Lemma \ref{lem:GaussSums} (5) to obtain 
\begin{align}
\label{eqn:GaussEven}
&G\left(2^{\ell+2\mu-\gamma} h M_0 \frac{M_0}{g_0}\frac{k_0}{g_0}, 2^{\ell+\mu+\varrho+1-\gamma} hr_0\frac{M_0}{g_0}; 2^{\kappa-\gamma}\right)\\
\nonumber& =\begin{cases}
1 &\text{if }\kappa\leq \ell+2\mu,\\
2&\text{if }\kappa=\ell+2\mu+1,\, \varrho=\mu-1,\\
2^{\frac{\kappa-\ell-2\mu}{2}} (1+i) \left(\frac{-2^{\ell+\kappa}}{h M_0 \frac{M_0}{g_0}\frac{k_0}{g_0}}\right) \varepsilon_{h M_0 \frac{M_0}{g_0}\frac{k_0}{g_0}} e^{-\frac{2\pi i hr_0^2}{2^{\kappa-\ell-2\mu}}\left[k_0\right]_{2^{\kappa-\ell-2\mu}} }&\text{if }\kappa\geq \ell+2\mu+2,\, \varrho=\mu,\\
0&\text{otherwise}.
\end{cases}
\end{align}
Plugging \eqref{eqn:GaussEven} into \eqref{eqn:tosplit2pow} and then simplifying yields that $G(a,b;c)$ equals
\[
\varepsilon_{\frac{k_0}{g_0}}\sqrt{kg_0}
\left(\tfrac{h g_0}{\frac{k_0}{g_0}}\right)
 \begin{cases}
2^{\frac{\kappa}{2}} \left(\tfrac{2^{\ell+\kappa} }{\frac{k_0}{g_0}}\right)  e^{-\frac{2\pi i   hr_0^2}{\frac{k_0}{g_0}}2^{\ell+2\varrho}\left[2^{\kappa}g_0\right]_{\frac{k_0}{g_0}}}  &\hspace{-.2cm}\begin{array}{l}\text{if }\kappa\leq \ell+2\mu\\ \quad \text{or }\kappa=\ell+2\mu+1,\, \varrho=\mu-1,\end{array}\\
2^{\frac{\ell+2\mu}{2}}(1+i) \varepsilon_{h M_0 \frac{M_0}{g_0}\frac{k_0}{g_0}}   \left(\frac{-2^{\ell+\kappa}}{h M_0 \frac{M_0}{g_0}\frac{k_0}{g_0}}\right) \left(\tfrac{2^{\ell+\kappa} }{\frac{k_0}{g_0}}\right) &\text{if }\kappa\geq \ell+2\mu+2,\, \varrho=\mu,\\
\hspace{0.1cm} \times e^{-\frac{2\pi i hr_0^2}{\frac{k_0}{g_0}}  2^{\ell+2\mu}\left[2^{\kappa}g_0\right]_{\frac{k_0}{g_0}} }e^{-\frac{2\pi i  hr_0^2}{2^{\kappa-\ell-2\mu}}\left[k_0\right]_{2^{\kappa-\ell-2\mu}}}&\\
0&\text{otherwise}.
\end{cases}
\]
To obtain the claim, we multiply by $e^{\frac{2\pi i 2^{\ell}hr^2}{k}}$ and simplify by using the Chinese Remainder Theorem to combine the exponentials. For example, if $\kappa\leq \ell+2\mu$ or ($\kappa=\ell+2\mu+1$ and $\varrho=\mu-1$), then the exponential becomes
\[
e^{\frac{2\pi i hr_0^2}{2^{\kappa-\delta}k_0} 2^{\ell+2\varrho-\delta}\left( 1-2^{\kappa} g_0 \left[2^{\kappa}g_0\right]_{\frac{k_0}{g_0}}\right)}. 
\]
Since $\gcd(g_0,\frac{k_0}{g_0})=g_1=1$ and $k_0$ is odd, to determine $1-2^{\kappa}g_0 \left[2^{\kappa}g_0\right]_{\frac{k_0}{g_0}} \pmod{2^{\kappa-\delta}k_0}$ the Chinese Remainder Theorem implies that it suffices to compute
\begin{align*}
1-2^{\kappa}g_0 \left[2^{\kappa}g_0\right]_{\frac{k_0}{g_0}}&\equiv 1 \pmod{g_0},\qquad 1-2^{\kappa}g_0 \left[2^{\kappa}g_0\right]_{\frac{k_0}{g_0}}\equiv 0 \pmod{\frac{k_0}{g_0}},\\
 1-2^{\kappa}g_0 \left[2^{\kappa}g_0\right]_{\frac{k_0}{g_0}}&\equiv 1 \pmod{2^{\kappa-\delta}}.
\end{align*}
Thus
\[
1-2^{\kappa}g_0 \left[g_0\right]_{\frac{k_0}{g_0}}\equiv \frac{k_0}{g_0}\left[\frac{k_0}{g_0}\right]_{2^{\kappa-\delta}g_0} \pmod{2^{\kappa-\delta}k_0}.
\]
So the exponential simplifies in this case as $e^{\frac{2\pi i  hr_0^2}{2^{\kappa-\delta}g_0}2^{\ell+2\varrho-\delta}[\frac{k_0}{g_0}]_{2^{\kappa-\delta}g_0}}$.

The remaining case $\kappa\geq \ell+2\mu+2$ and $\varrho=\mu$ follows by a similar but longer and more tedious calculation.
\end{proof}

\section{Growth towards the cusps of certain modular forms}\label{sec:growth} 
In this section, we determine the growth towards the cusps of theta functions $\Theta_{r,M,\bm{\alpha}}$ and certain (non-holomorphic) Eisenstein series. The purpose of this calculation is to compare the growth in order to determine the unique Eisenstein series $E_{r,M,\bm{\alpha}}$ in \eqref{eqn:thetadecompose} whose growth towards the cusps matches that of the theta function.

\subsection{Growth of the theta functions at the cusps}\label{sec:thetagrowth}

In order to obtain the Eisenstein series, we determine the growth of $\Theta_{r,M,\bm{\alpha}}$ towards all of the cusps, which follows by a straightforward calculation.

\begin{lemma}\label{lem:Thetacuspgrowth}
Let $m\in\N_{\geq 3}$ and $\bm{\alpha}\in\N^4$ be given. 
For $h\in\Z$ and $k\in\N$ with $\gcd(h,k)=1$, we have
\[ 
-\lim_{z\to 0^+} z^{2}\Theta_{r,M,\bm{\alpha}} \left( \frac hk+\frac{iz}{k}\right) = -\frac{1}{4k^2M^4\prod_{j=1}^4\sqrt{\alpha_j}} \prod_{j=1}^4 e^{\frac{2\pi i r^2h\a_j}{k}} G \left(h\alpha_jM^2,2hr\alpha_jM;k\right).
\]
\end{lemma}

We next use Lemma \ref{lem:G(a,b,c)eval} to evaluate the right-hand side of Lemma \ref{lem:Thetacuspgrowth}.  Since the theta function $\Theta_{r,M,\bm{\alpha}}$ only depends on $r$ modulo $M$, we may assume without loss of generality  that 
\[
\varrho=\ord_2(r)\leq \ord_2(M)=\mu
\]
 by replacing $r$ with $r+M$ in Lemma \ref{lem:Thetacuspgrowth} if $\varrho>\mu$.  A direct calculation gives the following.
\begin{corollary}\label{cor:Thetacuspgrowtheval}
Suppose that $h\in\Z$ and $k\in\N$ with $\gcd(h,k)=1$, $\bm{\alpha}=(1,2,4,8)$, $r\in\Z$, and $M\in\N$ with $\gcd(M,r)\in\{1,2,4\}$ and $\ord_2(r)\leq \ord_2(M)$. If $g_1\neq 1$ or $\varrho<\min(\mu,k-\ell-\mu)-1$, then
\[
-\lim_{z\to 0^+} z^{2}\Theta_{r,M,\bm{\alpha}} \left( \frac hk+\frac{iz}{k}\right)=0.
\]
If $g_1=1$ and $\varrho\ge\min(\mu,k-\ell-\mu)-1$, then, setting $\delta_0:=\min(\kappa,2\varrho)$,
\begin{align*}
-\lim_{z\to 0^+} z^{2}\Theta_{r,M,\bm{\alpha}}&\left( \frac hk+\frac{iz}{k}\right)\\
 &= \begin{cases}
- \frac{2^{2\kappa-4\mu-5}}{M_0^4}  g_0^2  e^{\frac{2\pi i hr_0^2}{2^{\kappa-\delta_0}g_0}  2^{2\varrho-\delta_0} 15 \left[\frac{k_0}{g_0}\right]_{2^{\kappa-\delta_0}g_0}}&\text{if }\kappa\leq 2\mu,\\
&\quad\text{or }\kappa=2\mu+1\text{ and }\varrho=\mu-1,\\
\frac{g_0^2}{M_0^4}   e^{\frac{2\pi i hr_0^2}{g_0}  15\left[2^{\kappa-2\mu}\frac{k_0}{g_0}\right]_{g_0}} &\text{if }\kappa\geq 2\mu+5\text{ and }\varrho=\mu,\\
0&\text{otherwise.}
\end{cases}
\end{align*}
\end{corollary}
\begin{remark}
Although the right-hand side of Corollary \ref{cor:Thetacuspgrowtheval} splits into a number of cases, we obtain an explicit element of the cyclotomic field $\Q(\zeta_{2^{j}g_0})$ for some $j\in\N_0$, where $\zeta_{\nu}:=e^{\frac{2\pi i}{\nu}}$. To use Corollary \ref{cor:Thetacuspgrowtheval} for practical purposes, one can evaluate the right-hand side of Corollary \ref{cor:Thetacuspgrowtheval} with a computer as an element of $\Q(\zeta_\nu)\cong \Q[x]/\left<f_{\nu}\right>$, where $f_\nu$ is the minimal polynomial of $\zeta_\nu$ over $\Q$, which is well-known to be 
\[
f_\nu(x)=\prod_{\substack{1\leq k\leq \nu\\ \gcd(k,\nu)=1}}\left(x-\zeta_\nu^k\right) = \prod_{d\mid \nu} \left(x^d-1\right)^{\mu\left(\frac{\nu}{d}\right)}.
\]
Here $\mu$ denotes the \begin{it}M\"obius $\mu$-function\end{it}.
\end{remark}

\subsection{Growth of Eisenstein series towards the cusps} 

The goal of this section is to obtain the growth of certain weight two Eisenstein series towards the cusps. These are formed by applying certain sieving and $V$-operators to the (non-holomorphic but modular) weight two Eisenstein series 
\[
\widehat{E}_2(\tau):=E_2(\tau)-\frac{3}{\pi v}, \qquad\text{ where }\qquad E_2(\tau):=1-24\sum_{n\in\N} \sigma(n) q^n
\]
with $\sigma(n):=\sigma_1(n)$. In light of Lemma \ref{lem:VSieveCommute}, we may furthermore always assume without loss of generality that sieving is applied before the $V$-operator. The growth towards the cusps of such functions is given in the following lemma.

\begin{lemma}\label{lem:ConstantE2SieveV}
Let $m\in\Z$ and $M_1,M_2\in\N$. Then for $h\in\Z$ and $k\in\N$ with $\gcd(h,k)=1$ we have
\[
-\lim_{z\to 0^+}z^2\widehat{E}_2\big|S_{M_1,m}\big|V_{M_2}\left(\frac{h}{k}+\frac{iz}{k}\right) = \frac{1}{M_1^3M_2^2}\sum_{j\pmod{M_1}} \gcd\left(hM_1M_2+jk,M_1k\right)^2 \zeta_{M_1}^{-jm}.
\]
\end{lemma}
 \begin{proof}
For a translation-invariant function $f$, we use the presentation 
\begin{equation*}
f|S_{M_1,m}(\tau)=\frac{1}{M_1} \sum_{j=0}^{M_1-1} \zeta_{M_1}^{-jm}f\lp \tau+\frac{j}{M_1} \rp.
\end{equation*}
Applying $V_{M_2}$ to this yields
\begin{equation*}
f|S_{M_1,m}\big|V_{M_2}(\tau)=\frac{1}{M_1} \sum_{j=0}^{M_1-1} \zeta_{M_1}^{-jm}f\lp M_2\tau+\frac{j}{M_1} \rp.
\end{equation*}
Plugging in $f=\widehat{E}_2$ and using the weight two modularity of $\widehat{E}_2$, the claim follows by a standard calculation. 
\end{proof}

\section{Eisenstein series component}\label{sec:Eisenstein}

In this section, we determine the Eisenstein series component $E_{r,M,\bm{\alpha}}$ in \eqref{eqn:thetadecompose}. 
\begin{proposition}\label{prop:Eisenstein} 
For $n\in\N$, we have the following. 
\begin{enumerate}[leftmargin=*,label={\rm(\arabic*)}]
	\item For $m=7$, we have $a_{7,10,(1,2,4,8)}(n)=0$ unless $n\equiv15\pmod{40}$, in which case we have
	\[
		a_{7,10,(1,2,4,8)}(n) = \frac{1}{240}\left(\sigma(n)-\sigma\left(\frac n5\right)\right).
	\]

	\item For $m=9$, we have $a_{9,14,(1,2,4,8)}(n)$ unless $n\equiv39\pmod{56}$, in which case we have 
	\[
		a_{9,14,(1,2,4,8)}(n) = \frac{1}{672} \sigma(n).
	\]

	\item For $m=10$, we have $a_{10,16,(1,2,4,8)}(n)$ unless $n\equiv28\pmod{64}$, in which case we have 
	\[
		a_{10,16,(1,2,4,8)}(n) = \frac{1}{256} \sigma\left(\frac n4\right).
	\]

	\item For $m=11$, we have $a_{11,18,(1,2,4,8)}(n)$ unless $n\equiv15\pmod{72}$, in which case we have 
	\[
		a_{11,18,(1,2,4,8)}(n) = \frac{1}{1728} \sigma(n).
	\]

	\item For $m=12$, we have $a_{12,20,(1,2,4,8)}(n)=0$ unless $80\mid n$, in which case we have
	\begin{multline*}
		\hspace{.4cm}a_{12,20,(1,2,4,8)}(n)=\frac{1}{120} \left(\sigma\left(\frac{n}{16}\right)-\sigma\left(\frac{n}{32}\right) - \sigma\left(\frac{n}{80}\right) +\sigma\left(\frac{n}{160}\right) +8\sigma\left(\frac{n}{256}\right)-32\sigma\left(\frac{n}{512}\right)\right.\\
		\left.-8\sigma\left(\frac{n}{1280}\right)+32\sigma\left(\frac{n}{2560}\right) \right).
	\end{multline*}

	\item For $m=13$, we have $a_{13,22,(1,2,4,8)}(n)=0$ unless $n\equiv71\pmod{88}$, in which case we have 
	\[
		a_{13,22,(1,2,4,8)}(n) = \frac{1}{2640} \sigma(n).
	\]

	\item For $m=14$, we have
	\[
		a_{14,24,(1,2,4,8)}(n) = 
		\begin{cases}
			\frac{1}{768} 		\left(\sigma\left(\frac{n}{4}\right)-\sigma\left(\frac{n}{12}\right)\right)&\text{if }n\equiv 60\pmod{96},\\
			0&\text{otherwise}.
		\end{cases}
	\]
\end{enumerate}

\end{proposition}
\begin{proof}\leavevmode\newline
(1) By comparing Fourier coefficients, we see that the identity is equivalent to 
\begin{equation}\label{eqn:guessmain2equiv}
E_{7,10,(1,2,4,8)}=-\frac{1}{5760}E_2\big|\left(1-V_5\right)\big|S_{40,15}.
\end{equation}
Lemma \ref{lem:ThetaModular} and \eqref{eqn:thetadecompose} give that
\[
	E_{7,10,\left(1,2,4,8\right)} \in M_2\Big(\Gamma_{3200,10}\Big),
\] 
while  Lemma \ref{lem:modularoperators} implies that
\[
	E_2\big|\left(1-V_5\right)\big|S_{40,15}\in M_2\left(\Gamma_1\left(1600\right)\right).
\]
Enumerating the cusps of $\Gamma_1(3200)$ (see \cite[Proposition 3.8.3]{DiamondShurman}), we then use a computer together with Lemma \ref{lem:ConstantE2SieveV} and Corollary \ref{cor:Thetacuspgrowtheval} to verify that the growth towards every cusp of both sides of \eqref{eqn:guessmain2equiv} agree, yielding the claim.

For the remaining cases, the argument is similar, but we provide the identities analogous to \eqref{eqn:guessmain2equiv} for the convenience of the reader. 

\noindent
(2) The claim is equivalent to 
\[
E_{9,14,(1,2,4,8)}=-\frac{1}{16128}E_2\big|S_{56,39}.
\]
(3) The claim is equivalent to 
\[
E_{10,16,(1,2,4,8)}=-\frac{1}{6144}E_2 \big|S_{16,7}\big|V_4.
\]
(4) The claim is equivalent to
\[
E_{11,18,(1,2,4,8)}=-\frac{1}{41472}E_2 \big|S_{72,15}.
\]
(5) The claim is equivalent to
\[
E_{12,20,(1,2,4,8)}=-\frac{1}{2880}E_2\big|\left(S_{5,0}-V_5\right)\big|\left(1- V_{2}+ 8V_{16}-32V_{32}\right)\big|V_{16}.
\]
(6) The claim is equivalent to
\[
E_{13,22,(1,2,4,8)}=-\frac{1}{63360}E_2 \big|S_{88,71}.
\]
(7) The claim is equivalent to
\[
E_{14,24,(1,2,4,8)}=-\frac{1}{18432} E_2 \big|\left(1-V_3\right)\big|S_{24,15}\big|V_4. \qedhere
\]
\end{proof}
As a corollary to Proposition \ref{prop:Eisenstein}, we obtain explicit lower bounds on the Fourier coefficients $a_{r,M,\bm{\alpha}}(n)$ in these special cases. 
\begin{corollary}\label{cor:EisensteinLower}
Let $n\in\N$.

\noindent
\begin{enumerate}[leftmargin=*,label={\rm(\arabic*)}]
\item If $n\equiv 15\pmod{40}$, then we have 
\[
a_{7,10,(1,2,4,8)}(n)\geq \frac{n}{240}.
\]
\item If $n\equiv 39\pmod{56}$, then we have
\[
a_{9,14,(1,2,4,8)}(n)\geq \frac{n}{672}.
\]
\item If $n\equiv 28\pmod{64}$, then we have 
\[
a_{10,16,(1,2,4,8)}(n)\geq \frac{n}{1024}.
\]
\item If $n\equiv 15\pmod{72}$, then we have
\[
a_{11,18,(1,2,4,8)}(n)\geq \frac{n}{1728}. 
\]
\item Assume that $80\mid n$ and write $n=2^a5^bc$ with $\gcd(10,c)=1$. We have 
\[
a_{12,20,(1,2,4,8)}(n)\geq \frac{5^b c}{120}\begin{cases} 2^{a-4}&\text{if }4\leq a\leq 7,\\  24&\text{if }a\geq 8.\end{cases}
\]
\item If $n\equiv 71\pmod{88}$, then we have 
\[
a_{13,22,(1,2,4,8)}(n)\geq \frac{n}{2640}.
\]
\item If $n\equiv 60\pmod{96}$, then we have 
\[
a_{14,24,(1,2,4,8)}(n)\geq \frac{n}{3072}.
\]

\end{enumerate}
\end{corollary}
\begin{proof}
For $m\neq 12$, the claims with the exception of (5) follow directly from Proposition \ref{prop:Eisenstein}. For (5), a direct simplification yields that the right-hand side of Proposition \ref{prop:Eisenstein} (5) simplifies as 
\[
\frac{5^b\sigma(c)}{120} \begin{cases} 2^{a-4}&\text{if }4\leq a\leq 7,\\ 24&\text{if }a\geq 8,\end{cases}
\]
which gives the claim.
\end{proof}

\section{Proof of Theorem \ref{thm:Sun1248sufflarge}}\label{sec:mainproof}

In this section, we prove  Theorem \ref{thm:Sun1248sufflarge}. The constants $C_m$ from the theorem statement may be found in Table \ref{tab:finalbounds}. 

\begin{proof}[Proof of Theorem \ref{thm:Sun1248sufflarge}]
We require the case $\ell=4$ of Lemma \ref{lem:KKT}. Since the inner sum only has a single term namely $m=0$ in this case, Lemma \ref{lem:KKT} simplifies as
\begin{multline}\label{eqn:cuspbound}
\|f_{r,M,\bm{\alpha}}\|^2\leq \frac{2\cdot 3^{6}}{\pi^{4}} \frac{M^{4}N_{\bm{\alpha}}^{2}}{\prod_{p\mid M^2N_{\bm{\alpha}}}\left(1-p^{-2}\right)}\\
\times \sum_{\delta\mid M^2N_{\bm{\alpha}}}\varphi\left(\frac{M^2N_{\bm{\alpha}}}{\delta}\right)\varphi(\delta)\frac{M^2N_{\bm{\alpha}}}{\delta}\left(\frac{\gcd(M^2,\delta)}{M^2}\right)^{4} 2
\left(\frac{27}{D_{\bm{\alpha}}} \frac{M^2N_{\bm{\alpha}}}{\pi} +16\right).
\end{multline}

For $\theta_{r,M,(1,2,4,8)}$, we obtain a lower bound for $a_{r,M,(1,2,4,8)}(n)$ (for $n$ in an appropriate congruence class) from Corollary \ref{cor:EisensteinLower} (see the third column of Table \ref{tab:bounds} for a list of the bounds for individual choices of $r$ and $M$).
 \begin{center}
\begin{table}[H]
\caption{Bounds for $a_{r,M,(1,2,4,8)}$, $\|f_{r,M,(1,2,4,8)}\|$, and $|b_{r,M,(1,2,4,8)}|$ \label{tab:bounds}}
\begin{tabular}{|c|c|c|c|c|}
\hline
$r$&$M$&Bound for $a_{r,M,(1,2,4,8)}$&Bound for $\|f_{r,M,(1,2,4,8)}\|$&Bound for $|b_{r,M,(1,2,4,8)}|^{\vphantom{a^b}}$\\[1ex]
\hline
$7$&$10$&$\frac{n}{240}$ &$8.11\cdot 10^{14}$ & $3.41\cdot 10^{30^{\vphantom{a^b}}} n^{\frac{3}{5}}$\\[1ex]
\hline
$9$&$14$&$\frac{n}{672}$ &$1.03\cdot 10^{16}$ & $3.48\cdot 10^{31^{\vphantom{a^b}}} n^{\frac{3}{5}}$\\[1ex]
\hline
$10$&$16$&$\frac{n}{1024}$ &$3.2\cdot 10^{16}$ &$9.98\cdot 10^{31^{\vphantom{a^b}}} n^{\frac{3}{5}}$ \\[1ex]
\hline
$11$&$18$&$\frac{n}{1728}$ &$6.1\cdot 10^{16}$ &$1.52\cdot 10^{32^{\vphantom{a^b}}} n^{\frac{3}{5}}$ \\[1ex]
\hline
$12$&$20$ &$\frac{n}{1920}$ &$1.49\cdot 10^{17}$ & $3.69\cdot 10^{32^{\vphantom{a^b}}}n^{\frac{3}{5}}$\\[1ex]
\hline
$13$& $22$& $\frac{n}{2640}$ &$2.55\cdot 10^{17}$ & $6.96\cdot 10^{32^{\vphantom{a^b}}}n^{\frac{3}{5}}$\\[1ex]
\hline
$14$& $24$ &$\frac{n}{3072}$ & $5.63\cdot 10^{17}$ & $1.09\cdot 10^{33^{\vphantom{a^b}}} n^{\frac{3}{5}}$ \\[1ex]
\hline
\end{tabular}
\end{table}
\end{center}

 Computing the constants in \eqref{eqn:cuspbound} explicitly for fixed $M$ yields an upper bound for $\|f_{r,M,(1,2,4,8)}\|^2$ (see the fourth column of Table \ref{tab:bounds} for the explicit bounds), which plugged into Lemma \ref{lem:BanerjeeKane} yields an upper bound for $|b_{r,M,(1,2,4,8)}(n)|$ (see the final column of Table \ref{tab:bounds} for the explicit bounds). Plugging the bounds for $a_{r,M,(1,2,4,8)}(n)$ and $|b_{r,M,(1,2,4,8)}(n)|$ into  \eqref{eqn:sabdecompose}, we see that $s_{r,M,(1,2,4,8)}(n)>0$ for $n$ sufficiently large in an appropriate congruence class (see Table \ref{tab:finalbounds} for the explicit constants). 
\begin{center}
\begin{table}[H]
\caption{Bounds on $n$ for $s_{m,2(m-2),(1,2,4,8)}(n)>0$ and $r_{m,(1,2,4,8)}>0$ \label{tab:finalbounds}}
\begin{tabular}{|c|c|c|}
\hline
$m$&Bound for $s_{m,2(m-2),(1,2,4,8)}(n)>0$&Bound $C_m$ for $r_{m,2(m-2),(1,2,4,8)}(n)>0^{\vphantom{a^b}}$\\[1ex]
\hline
$7$ & $1.92\cdot 10^{82^{\vphantom{a^b}}}$ &$4.8\cdot 10^{80}$ \\
\hline
$9$ & $8.38\cdot 10^{85^{\vphantom{a^b}}}$ &$1.5\cdot 10^{84}$ \\
\hline
$10$& $3.41\cdot 10^{87^{\vphantom{a^b}}}$& $5.33\cdot 10^{85}$ \\
\hline
$11$& $3.55\cdot 10^{88^{\vphantom{a^b}}}$&$4.93\cdot 10^{86}$ \\
\hline
$12$& $4.25\cdot 10^{89^{\vphantom{a^b}}}$  &$5.31\cdot 10^{87}$ \\
\hline
$13$& $4.57\cdot 10^{90^{\vphantom{a^b}}}$&$5.19\cdot 10^{88}$ \\
\hline
$14$& $2.04\cdot 10^{91^{\vphantom{a^b}}}$& $2.13\cdot 10^{89}$ \\
\hline
\end{tabular}
\end{table}
\end{center}
We then conclude that $r_{m,(1,2,4,8)}>0$ for $n$ sufficiently large by using \eqref{eqn:r*s*rel}, yielding the claim. \qedhere
\end{proof}

In order to explain why it is sufficient to assume that $n\not\equiv 4\pmod{16}$ for $m=12$ in Theorem \ref{thm:Sun1248sufflarge}, we require the following lemma combined with \eqref{eqn:r*s*rel}.
\begin{lemma}\label{lem:p122power}
Let $n\in\N$ be given. If the equation
\[
x_1^2+2x_2^2+4x_3^2+8x_4^2=n
\]
 is solvable with $x_j\equiv 12\pmod{20}$, then the equation 
\[
x_1^2+2x_2^2+4x_3^2+8x_4^2=256n
\]
is also solvable with $x_j\equiv 12\pmod{20}$.
\end{lemma}

\end{document}